\documentclass[graybox]{svmult}

\usepackage{type1cm}        % activate if the above 3 fonts are
\usepackage{graphicx}        % standard LaTeX graphics tool
\usepackage{multicol}        % used for the two-column index
\usepackage[bottom]{footmisc}% places footnotes at page bottom

\usepackage{newtxtext}       % 
\usepackage{newtxmath}       % selects Times Roman as basic font
\usepackage{url}
\newcommand{\pOt}{{\partial \Omega_t}}
\newcommand{\Ht}{H^{3/2}(S^{n-1})}
\newcommand{\Hp}{H^{1/2}(S^{n-1})}
\newcommand{\Hm}{H^{-1/2}(S^{n-1})}
\newcommand{\Hh}{\Hp \oplus \Hm}

\newcommand{\wtf}{\tilde f}
\newcommand{\wtg}{\tilde g}
\newcommand{\wth}{\tilde h}

\newcommand{\p}{\partial}

\newcommand{\bbC}{{\mathbb{C}}}

\newcommand{\bbN}{{\mathbb{N}}}

\newcommand{\bbR}{{\mathbb{R}}}

\newcommand{\bbZ}{{\mathbb{Z}}}

\newcommand{\cD}{{\mathcal D}}

\newcommand{\cH}{{\mathcal H}}

\DeclareMathOperator{\Tr}{Tr}

\newcommand{\beq}{\begin{equation}}
\newcommand{\enq}{\end{equation}}

\begin{document}

\title*{Exponential dichotomies for elliptic PDE on radial domains}
\author{M. Beck, G. Cox, C. Jones, Y. Latushkin and A. Sukhtayev}
\institute{Margaret Beck \at Department of Mathematics and Statistics, Boston University, Boston, MA 02215, USA, \email{mabeck@bu.edu}
\and Graham Cox \at Department of Mathematics and Statistics, Memorial University of Newfoundland, St. John's, NL A1C 5S7, Canada, \email{gcox@mun.ca}
\and Christopher Jones \at Department of Mathematics, University of North Carolina at Chapel Hill, Chapel Hill, NC 27599, USA, \email{ckrtj@email.unc.edu}
\and Yuri Latushkin \at Department of Mathematics, University of Missouri, Columbia, MO 65211, USA, \email{latushkiny@missouri.edu}
\and Alim Sukhtayev \at Department of Mathematics, Miami University, Oxford, OH 45056, USA, \email{sukhtaa@miamioh.edu}
}
\maketitle

% abstract (without star)
\abstract{It was recently shown by the authors that a semilinear elliptic equation can be represented as an infinite-dimensional dynamical system in terms of boundary data on a shrinking one-parameter family of domains. The resulting system is ill-posed, in the sense that solutions do not typically exist forward or backward in time. In this paper we consider a radial family of domains and prove that the linearized system admits an exponential dichotomy, with the unstable subspace corresponding to the boundary data of weak solutions to the linear PDE. This generalizes the spatial dynamics approach, which applies to infinite cylindrical (channel) domains, and also generalizes previous work on radial domains as we impose no symmetry assumptions on the equation or its solutions.}

\section{Introduction}
The fundamental idea of spatial dynamics is to write a partial differential equation on a cylindrical domain $\Omega = \bbR \times \Omega' \subset \bbR^n$ as an ordinary differential equation with respect to the longitudinal variable $x \in \bbR$. For instance, $\Delta u + F(x,y,u) = 0$ becomes
\[
	\frac{du}{dx} = v, \quad \frac{dv}{dx} = -F(x,y,u) - \Delta_{\Omega'} u,
\]
where $(x,y) \in \bbR \times \Omega'$ and $\Delta_{\Omega'}$ denotes the Laplacian on the cross-section $\Omega' \subset \bbR^{n-1}$. This idea first appeared in \cite{K82}; see also \cite{A84,BSZ10,DSSS09,G86,LP08,M86,PSS97,S02,SS01,S03} and references therein.

In \cite{BCJLS2} we extended this ODE--PDE correspondence to semi-linear elliptic equations on bounded domains. Assuming $\Omega$ is smoothly deformed through a one-parameter family $\Omega_t$, we obtain a dynamical system satisfied by the boundary data of solutions to $\Delta u + F(x,u) = 0$ on $\pOt$.

In the current paper we start to investigate the application of dynamical systems methodology to the resulting system of equations, which we call the Spatial Evolutionary System (SES). In particular, we construct an exponential dichotomy, and prove that the unstable subspace coincides with the space of boundary data for weak solutions to the PDE.

Our results are valid for systems of equations; functions are thus assumed to take values in $\bbC^N$ unless stated otherwise. We abbreviate $H^s(S^{n-1};\bbC^N) = H^s(S^{n-1})$ etc.

Suppose $u$ is a smooth solution to the linear elliptic system
\begin{align}\label{PDE}
	\Delta u = V u
\end{align}
on $\bbR^n$, where $V$ is an $N\times N$ matrix-valued function.  Writing $u = u(r,\theta)$ in terms of generalized polar coordinates $(r,\theta) \in (0,\infty) \times S^{n-1}$, we define the functions
\begin{align*}
	f(t) := u(t,\cdot), \quad
	g(t) := \frac{\p u}{\p r}(t,\cdot),
\end{align*}
which are in $C^\infty(S^{n-1})$ for $t>0$, and combine these to form the trace
\begin{equation}
    \Tr_t u := (f(t),g(t)).
\end{equation}
Using the fact that
\[
\Delta u = \frac{\p^2 u}{\p r^2} + \frac{n-1}{r} \frac{\p u}{\p r} + \frac{1}{r^2} \Delta_{S^{n-1}} u,
\]
a direct computation shows that for all $t>0$, $f$ and $g$ satisfy the linear system
\begin{align}\label{SES}
	\frac{d}{dt} \begin{pmatrix} f \\ g \end{pmatrix} = 
	\begin{pmatrix} 0 & 1 \\ V_t - t^{-2} \Delta_{S^{n-1}} \ \ & \ \ -(n-1) t^{-1} \end{pmatrix} \begin{pmatrix} f \\ g \end{pmatrix},
\end{align}
where $V_t := V(t,\cdot)$ and $\Delta_{S^{n-1}}$ is the Laplace--Beltrami operator on the sphere.  

In \cite{BCJLS2} it was shown that the equivalence between \eqref{PDE} and \eqref{SES} extends to  weak $H^1$ solutions. To state this precisely, consider the Hilbert spaces
\begin{align*}%\label{def:H}
	\cH = \Hp \oplus \Hm, 
	\quad \cH^1 = \Ht \oplus \Hp.
\end{align*}
The results can then be summarized as follows, where $B_T$ denotes the open ball of radius $T$.

\begin{theorem}\label{thm:ODEPDE}
Let $u \in H^1(B_T)$ be a weak solution to \eqref{PDE} for some $T>0$. Then $(f(t),g(t)) = \Tr_t u$ satisfies the regularity conditions
\begin{align}\label{fgreg}
	(f,g) \in C^0\big((0,T), \cH^1\big) \cap C^1\big((0,T),\cH\big) \cap C^0\big((0,T],\cH\big),
\end{align}
solves \eqref{SES} for $0 < t < T$, and has $\|f(t)\|_{\Hp} + \|g(t)\|_{\Hm}$ bounded near $t=0$.

On the other hand, if $(f,g)$ satisfies \eqref{fgreg}, solves \eqref{SES} for $0 < t < T$, and has $t^p \|f(t)\|_{\Hp} + t^{n-p-1} \|g(t)\|_{\Hm}$ bounded near $t=0$ for some $p \in (0,n/2)$, then there exists a weak solution $u \in H^1(B_T)$ to \eqref{PDE} with $\Tr_t u = (f(t),g(t))$ for all $t \in (0,T)$.
\end{theorem}

This equivalence also extends to semilinear equations on non-radial domains; see \cite{BCJLS2} for the general statement.

The system \eqref{SES} is ill-posed, in the sense that solutions do not necessarily exist forward (or backward) in time for given initial (or terminal) data at prescribed at time $t_0 > 0$. However, we will prove that $\cH$ splits into two infinite-dimensional subspaces for which the system admits solutions forward and backward in time, respectively. This property is described using the language of exponential dichotomies. The system \eqref{SES} does not admit an exponential dichotomy in the strict sense. Rather, a dichotomy exists for a suitably rescaled and reparameterized system of equations. 

We let $t = e^\tau$ and then define
\begin{align}\label{scaling}
	\wtf(\tau) = e^{\alpha\tau} f(e^\tau), \quad \wtg(\tau) = e^{(1+\alpha)\tau} g(e^\tau)
\end{align}
for some constant $\alpha$ to be determined. A direct computation shows that if $(f,g)$ solves \eqref{SES}, then
\begin{align}\label{RSES}
	\frac{d}{d\tau} \begin{pmatrix} \wtf \\ \wtg \end{pmatrix} = 
	\begin{pmatrix} \alpha & 1\\ e^{2\tau} V_{e^\tau} - \Delta_{S^{n-1}}  \ \ & \ \ \alpha + 2 - n \end{pmatrix} \begin{pmatrix} \wtf \\ \wtg \end{pmatrix}
\end{align}
for all $\tau \in \bbR$. For convenience we let $\wth = (\wtf,\wtg)$.

Our main result is that \eqref{RSES} has an exponential dichotomy on the half line $(-\infty,0]$ for most values of $\alpha$. Let
\begin{align}\label{Aspec}
	\Sigma(n) = \big((-\infty,2-n] \cup [0,\infty)\big) \cap \bbZ,
\end{align}
so that $\Sigma(2) = \Sigma(3) = \bbZ$, $\Sigma(4) = \bbZ \setminus \{-1\}$, etc. We also define the interpolation spaces
\[
    \cH^\beta = H^{1/2 + \beta}(S^{n-1}) \oplus H^{-1/2 + \beta}(S^{n-1}),
\]
so that $\cH^0 = \cH$ and $\cH^1$ agrees with the definition given above.

\begin{theorem}\label{thm:dichotomy}
If $-\alpha \notin \Sigma(n)$ and $V \in C^{0,\gamma}(B_1)$ for some $\gamma \in (0,1)$, then for each $\beta \in [0,1)$ there exists a H\"older continuous family of projections $P^u\colon (-\infty,0] \to B(\cH^\beta)$ and constants $K, \eta^u, \eta^s > 0$ such that, for every $\tau_0 \leq 0$ and $z \in \cH^\beta$ there exists a solution $\wth^u(\tau;\tau_0,z)$ of \eqref{RSES}, defined for $\tau \leq \tau_0$, such that
\begin{itemize}
	\item $\wth^u(\tau_0;\tau_0,z) = P^u(\tau_0) z$,
	\item $\| \wth^u(\tau; \tau_0,z) \|_{\cH^\beta} \leq K e^{\eta^u(\tau-\tau_0)} \|z\|_{\cH^\beta}$ for all  $\tau \leq \tau_0$,
	\item $\wth^u(\tau; \tau_0,z) \in R(P^u(\tau))$ for all $\tau \leq \tau_0$,
\end{itemize}
and a solution $\wth^s(\tau;\tau_0,z)$ of \eqref{RSES}, defined for $\tau_0 \leq \tau \leq 0$, such that
\begin{itemize}
	\item $\wth^s(\tau_0;\tau_0,z) = P^s(\tau_0) z$,
	\item $\| \wth^s(\tau; \tau_0,z) \|_{\cH^\beta} \leq K e^{\eta^s(\tau_0-\tau)} \|z\|_{\cH^\beta}$ for all $\tau_0 \leq \tau \leq 0$,
	\item $\wth^s(\tau; \tau_0,z) \in R(P^s(\tau))$ for all $\tau_0 \leq \tau \leq 0$,
\end{itemize}
where $P^s(\tau) = I - P^u(\tau)$.
\end{theorem}

We will see below that the exponential dichotomy on $(-\infty,0]$ carries information about bounded solutions to the linear PDE \eqref{PDE} on the unit ball, $B_1$. By the same method we can also obtain an exponential dichotomy on $(-\infty,\log T]$ for any $T>0$, corresponding to the PDE on the ball $B_T$.

The exponential dichotomy can also be described in terms of operators $\Phi^{s}(\tau,\tau_0)$ and $\Phi^{u}(\tau,\tau_0)$, defined by
\begin{align}\label{Phidef}
    \Phi^{s,u}(\tau,\tau_0) z = \wth^{s,u}(\tau;\tau_0,z)
\end{align} 
for $z \in \cH^\beta$, so that $\Phi^{s,u}(\tau_0,\tau_0) = P^{s,u}(\tau_0)$. Note that $\Phi^u(\tau,\tau_0) z$ is defined for $\tau \leq \tau_0 \leq 0$ and $\Phi^s(\tau,\tau_0) z$ is defined for $\tau_0 \leq \tau \leq 0$. From Theorem \ref{thm:dichotomy} we have the estimates
\[
    \| \Phi^u(\tau,\tau_0) z \|_{\cH^\beta} \leq K e^{\eta^u(\tau - \tau_0)} \|z\|_{\cH^\beta}, \quad \tau \leq \tau_0
\]
and
\[
    \| \Phi^s(\tau,\tau_0) z \|_{\cH^\beta} \leq K e^{\eta^s(\tau_0 - \tau)} \|z\|_{\cH^\beta}, \quad \tau_0 \leq \tau \leq 0.
\]

The precise growth and decay rates depend on $\alpha$. We will see below that it is convenient to choose $0 < \alpha < n-2$ (assuming $n>2$), in which case a dichotomy will exist for any numbers $\eta^u$ and $\eta^s$ satisfying $0 \leq \eta^u < \alpha$ and $0 \leq \eta^s < n-2-\alpha$.

To simplify the exposition we now assume $\beta=0$. For any $\tau \leq 0$ we define the unstable subspace $\widetilde E^u(\tau) = R(P^u(\tau))$, and then let
\begin{align}
	E^u(t) = \left\{ \left( t^{-\alpha} \wtf(\log t), t^{-1 - \alpha} \wtg(\log t) \right) : \big(\wtf(\log t), \wtg(\log t) \big) \in \widetilde E^u(\log t) \right\}
\end{align}
for $0 <t \leq 1$.

As in \cite{BCJLS2}, for an appropriate choice of $\alpha$ we have that $E^u(t)$ corresponds to the space of boundary data of weak solutions to \eqref{PDE} on the ball $B_t$. For $t>0$ let
\[
	K_t = \{u \in H^1(B_t) : \Delta u = V u \text{ on } B_t \},
\]
where the equality $\Delta u = V u$ is meant in a distributional sense. Since $K_t$ is a subset of $\{u \in H^1(B_t) : \Delta u \in L^2(B_t)\}$, the trace map $\Tr_t$ can be applied, and we have $\Tr_t u \in \Hh$ for each $u \in K_t$. We thus define
\[
	\Tr_t(K_t) = \{\Tr_t u : u \in K_t\} \subset \cH.
\]

The following result is then an immediate consequence of Theorem \ref{thm:dichotomy} and \cite[Theorem 3.10]{BCJLS2}.

\begin{theorem}\label{thm:unstable}
Assume, in addition to the hypotheses of Theorem \ref{thm:dichotomy}, that $V$ is smooth in a neighborhood of the origin. If 
\begin{align}\label{alpharange}
	-\eta^s < \alpha < \eta^u + \frac{n}{2} - 1,
\end{align}
then $E^u(t) = \Tr_t(K_t)$ for each $t>0$.
\end{theorem}

To verify \eqref{alpharange} we must understand the dependence of $\eta^u$ and $\eta^s$ on $\alpha$. When $n>2$ there is always an $\alpha$ for which \eqref{alpharange} is satisfied.

\begin{corollary}\label{cor:alphagap}
If $n>2$ and $0 < \alpha < n - 2$, then $E^u(t) = \Tr_t(K_t)$ for each $t > 0$.
\end{corollary}

On the other hand, no such $\alpha$ exists when $n=2$. This observation, which will be proved in Section \ref{sec:proofs} below, was also made in \cite[Remark 2.1]{BCJLS2}. Below we provide a different (but equivalent) explanation in terms of the spectrum of the limiting (as $\tau \to -\infty$) operator in \eqref{RSES}. For harmonic functions (i.e. when $V=0$) it can be shown that $E^u(t) = \Tr_t(K_t)$ if and only if $0 < \alpha < n-2$. This is proved in Section \ref{sec:harmonic} for $n=3$, and follows from a similar computation for other $n$.

\subsection*{Outline of the paper}
The remainder of the paper is organized as follows. In Section \ref{sec:dichotomy} we construct the half-line exponential dichotomy, proving Theorem \ref{thm:dichotomy} and Corollary \ref{cor:alphagap}. In Section \ref{sec:harmonic} we illustrate our results for the case of harmonic functions in $\bbR^3$, where the dichotomy projections can be found explicitly. Finally, in Section \ref{sec:apply} we use the exponential dichotomy to reformulate a nonlinear elliptic equation as a fixed point problem for an integral equation, and give a dynamical interpretation of a linear eigenvalue problem.

\section{Construction of the exponential dichotomy}\label{sec:dichotomy}
We prove Theorem \ref{thm:dichotomy} using the results of \cite{PSS97}. We start by decomposing the right-hand side of \eqref{RSES} as
\begin{align}\label{Adecomp}
	\begin{pmatrix} \alpha & 1\\ e^{2\tau} V_{e^\tau} - \Delta_{S^{n-1}}  \ \ & \ \ \alpha + 2 - n \end{pmatrix} 
\	&= \underbrace{\begin{pmatrix} \alpha & 1\\ - \Delta_{S^{n-1}}  \ \ & \ \ \alpha + 2 - n \end{pmatrix}}_{A}  + \underbrace{\begin{pmatrix} 0 & 0 \\ e^{2\tau} V_{e^\tau} & 0 \end{pmatrix}}_{B(\tau)}
\end{align}
where $A$ is an unbounded operator on $\cH = \Hp \oplus \Hm$ with domain $\cH^1 = \Ht \oplus \Hm$, and $B(\tau)$ is a bounded operator on $\cH$.

Before proceeding, we remark on the definition of the fractional Sobolev spaces appearing in our analysis. Following \cite{M00}, we define $H^s(S^{n-1})$ through local coordinate charts and a partition of unity. On the other hand, following \cite{S83,T92}, one can also define
\[
   \widetilde H^s = \left\{f \in L^2 : f = (I - \Delta)^{-s/2} g \text{ for some } g \in L^2 \right\}, \quad \|f\|_{\widetilde H^s} = \|g\|_{L^2}
\]
for $s>0$ and 
\begin{align*}
    \widetilde H^s &= \left\{f \in \mathcal D : f = (I - \Delta)^\ell g \text{ for some } g \in \widetilde H^{2\ell+s} \text{ with } \ell \in \bbN \text{ and } 2\ell + s > 0 \right\}, \\
    &\|f\|_{\widetilde H^s} = \|g\|_{\widetilde H^{2\ell+s}}
\end{align*}
for $s<0$, where $\mathcal D$ denotes the space of distributions. In either case we have that
\begin{equation}\label{Hsequiv}
    \|f\|_{\widetilde H^s}^2 = \sum_{k=1}^\infty (1 + \lambda_k)^s |c_k|^2
\end{equation}
where $(\lambda_k,\phi_k)$ are the eigenvalues and eigenfunctions of $-\Delta$ and $c_k = \left<f, \phi_k\right>$. This is equivalent to the local definition (see, for instance, \cite[Theorem 3.9]{GS13}), so we can use the $H^s$ and $\widetilde H^s$ norms interchangeably.

When $s=-\frac12$ we choose $\ell=1$, so that $2\ell + s = \frac32$, and thus obtain $\|f\|_{\widetilde H^{-1/2}} = \| g \|_{\widetilde H^{3/2}}$, where $g \in \widetilde H^{3/2}$ solves $(I - \Delta)g = f$. In particular, this implies
\begin{equation}\label{elliptic}
    \| g \|_{\widetilde H^{3/2}} \leq \| \Delta g\|_{\widetilde H^{-1/2}} + \|g\|_{\widetilde H^{-1/2}}
\end{equation}
for any $g \in \widetilde H^{3/2}$.

\subsection{The limiting operator}
In this section we describe the relevant properties of $A$.

\begin{lemma}\label{lem:compact}
$A$ is a closed operator with compact resolvent.
\end{lemma}

\begin{proof}
We first prove that the resolvent set of $A$ is nonempty. First consider
\begin{equation}\label{A0def}
	A_0 := \begin{pmatrix} 0 & 1 \\ - \Delta_{S^{n-1}} & 0  \end{pmatrix}.
\end{equation}
A direct computation shows that
\begin{equation}\label{A0resolvent}
	(A_0 - i\mu)^{-1} = \begin{pmatrix} i\mu D(\mu)^{-1} & D(\mu)^{-1} \\
	-\Delta_{S^{n-1}} D(\mu)^{-1} & i\mu D(\mu)^{-1} \end{pmatrix}
\end{equation}
where $D(\mu) := -\Delta_{S^{n-1}} + \mu^2$ is invertible for any $\mu \neq 0$. In particular, this implies the spectrum of $A_0$ is real. Since
\[
	A - A_0 = \begin{pmatrix} \alpha & 0 \\ 0 & 2 + \alpha - n \end{pmatrix}
\]
is a bounded operator on $\cH$, the spectrum of $A$ is contained in a bounded strip around the real axis, and hence the resolvent set is nonempty. The compactness of the resolvent operator now follows from the compactness of the embedding $\cH^1 \hookrightarrow \cH$.

We next prove that $A$ is closed. It suffices to prove that $A_0$ is closed, since $A - A_0$ is bounded. To that end, let $(f_k,g_k)$ be a sequence in $\Ht \oplus \Hp$ such that $(f_k,g_k) \to (f,g)$ in $\cH$
and $A_0(f_k,g_k) \to (F,G)$ in $\cH$. This means $g_k \longrightarrow F$ in $\Hp$ and $-\Delta_{S^{n-1}} f_k  \longrightarrow G$ in $\Hm$. From \eqref{elliptic} we have the estimate
\[
	\|f\|_{\Ht} \leq C \left( \big\| \Delta_{S^{n-1}} f \big\|_{\Hm} + \|f\|_{\Hm} \right)
\]
for all $f \in \Ht$. Since $f_k \to f$ in $\Hp$ and $\Delta_{S^{n-1}} f_k \to - G$ in $\Hm$, the estimate implies that $f_k \to f$ in $\Ht$. Therefore, $(f_k,g_k) \to (f,g)$ in $\Ht \oplus \Hp$, and so $A_0(f_k,g_k) \to A_0(f,g) = (F,G)$ in $\Hp \oplus \Hm$. This completes the proof that $A_0$ (and hence $A$) is closed.
\end{proof}

We now compute the spectrum of $A$.

\begin{lemma}\label{lem:spectrum}
The spectrum of $A$ is $\alpha + \Sigma(n)$, where $\Sigma(n)$ is the set defined in \eqref{Aspec}.
\end{lemma}

\begin{proof}
It suffices to show that the spectrum is $\Sigma(n)$ when $\alpha=0$. Since $A$ has compact resolvent, the spectrum is discrete and contains only eigenvalues. For $\alpha=0$ the eigenvalue equation is 
\[
	\begin{pmatrix} 0 & 1 \\ -\Delta_{S^{n-1}} \ &\  2-n \end{pmatrix} \begin{pmatrix} f \\ g \end{pmatrix} = 
	\nu  \begin{pmatrix} f \\ g \end{pmatrix}
\]
hence $g = \nu f$ and $-\Delta_{S^{n-1}} f + (2-n)g = \nu g$, which we combine to obtain
\[
	-\Delta_{S^{n-1}} f = \nu(\nu + n-2) f.
\]
The distinct eigenvalues of $-\Delta_{S^{n-1}}$ are of the form $l(l + n-2)$ for $l \in \bbN \cup\{0\}$. Setting $\nu(\nu+n-2) = l(l + n-2)$, we obtain $\nu = l, 2-n-l$ as claimed.
\end{proof}

Finally, we prove a resolvent estimate for $A$.

\begin{lemma}\label{lem:resolvent}
For $-\alpha \notin \Sigma(n)$ there exists $C>0$ such that
\begin{align}\label{eqn:resolvent}
	\big\| (A - i\mu)^{-1} \big\|_{B(\cH)} \leq \frac{C}{1 + |\mu|}
\end{align}
for all $\mu \in \bbR$.
\end{lemma}

\begin{proof}
From Lemma \ref{lem:spectrum}, the hypothesis on $\alpha$ guarantees $A-i\mu$ is boundedly invertible for any $\mu \in \bbR$, so we just need to prove that \eqref{eqn:resolvent} holds when $|\mu|$ is sufficiently large.

We next observe that it is enough to prove the estimate for the operator $A_0$ defined in \eqref{A0def}. If the estimate holds for $A_0$ we can choose $\mu$ large enough that $\| (A_0 - i\mu)^{-1}(A - A_0) \|_{B(\cH)} \leq 1/2$, since $A - A_0 \in B(\cH)$. This implies  $I + (A_0 - i\mu)^{-1}(A - A_0)$ is invertible, with
\[
	\big\| \big(I + (A_0 - i\mu)^{-1}(A - A_0)\big)^{-1} \big\|_{B(\cH)} \leq \sum_{k=0}^\infty \left(\frac{1}{2}\right)^k = 2.
\]
Writing
$
	A - i\mu = (A_0 - i\mu) \big(I + (A_0 - i\mu)^{-1}(A-A_0) \big),
$
we thus obtain $\|(A - i\mu)^{-1}\|_{B(\cH)} \leq 2C / (1 + |\mu|)$.

It remains to prove the resolvent estimate \eqref{eqn:resolvent} for $A_0$ when $|\mu|$ is large. The resolvent is given by \eqref{A0resolvent}. Therefore it suffices to prove the estimates
\begin{align*}
    \big\| D(\mu)^{-1} \big\|_{B(\Hm)} \leq \frac{C}{1 + \mu^2} \\
	\big\| D(\mu)^{-1} \big\|_{B(\Hp)} \leq \frac{C}{1 + \mu^2} \\
	\big\| D(\mu)^{-1} \big\|_{B(\Hm,\Hp)} \leq \frac{C}{1 + |\mu|} \\
	\big\| D(\mu)^{-1} \big\|_{B(\Hp,\Ht)} \leq \frac{C}{1 + |\mu|}
\end{align*}
for sufficiently large $|\mu|$.

Letting $(\lambda_k)$ denote the eigenvalues of $-\Delta_{S^{n-1}}$, and $(\phi_k)$ the corresponding eigenfunctions, we can compute the $H^s$ norm of $f$ by
\begin{align}
	\|f\|^2_{H^s(S^{n-1})} = \sum_k (1+\lambda_k)^s |c_k|^2,
\end{align}
where $c_k = \left<f, \phi_k\right>$. For smooth $f$ we have
\begin{align*}
	\big\| D(\mu) f \big\|^2_{H^s(S^{n-1})} = \sum_k (1+\lambda_k)^s (\lambda_k + \mu^2)^2 |c_k|^2.
\end{align*}
Using the inequality $ (\lambda_k + \mu^2)^2 \geq \mu^4$, we obtain
\begin{align*}
	\big\| D(\mu) f \big\|^2_{H^s(S^{n-1})}
	\geq \mu^4 \sum_k (1 + \lambda_k)^s |c_k|^2 = \mu^4 \| f \|_{H^s(S^{n-1})}^2.
\end{align*}
Similarly, assuming without loss of generality that $|\mu| \geq 1$, we find that
\[
    (\lambda_k + \mu^2)^2 = \lambda_k^2 + 2 \lambda_k \mu^2 + \mu^4  \geq \lambda_k \mu^2 + \mu^2 = \mu^2 (1 + \lambda_k) \geq \frac12 (1 + |\mu|)^2(1 + \lambda_k)
\]
and hence
\begin{align*}
	\big\| D(\mu) f \big\|^2_{H^s(S^{n-1})} 
	\geq \frac12 (1 + |\mu|)^2 \sum_k (1+\lambda_k)^{s+1} |c_k|^2 = \frac12 (1 + |\mu|)^2 \| f \|_{H^{s+1}(S^{n-1})}^2,
\end{align*}
which completes the proof.
\end{proof}

\subsection{The perturbation}
We now establish the required continuity and decay properties of the perturbation $B$.

\begin{lemma}\label{lem:pert}
$B(\cdot) \in C^{0,\gamma}\big((-\infty,0],B(\cH^\beta,\cH)\big)$ and $\|B(\tau)\|_{B(\cH^\beta,\cH)} \leq C e^{2\tau}$ for $\tau \leq 0$.
\end{lemma}

\begin{proof}

From the definition of $B(\tau)$ in \eqref{Adecomp} we obtain
\begin{align*}
	\|B(\tau)\|_{B(\cH^\beta,\cH)} = t^2 \|V_t\|_{B(H^{1/2+\beta}(S^{n-1}),\Hm)},% \leq C t^2 \| V_t \|_{L^\infty(\pO)}
\end{align*}
where $V_t$ denotes the operator on $H^{1/2+\beta}(S^{n-1})$ that is multiplication by $V_t$ followed by inclusion into $\Hm$. For any $f \in \Hp$ we have
\begin{align*}
	\| V_t f\|_{\Hm} = \sup_{g\neq0} \frac{|\left<V_t f,g\right>|}{\|g\|_{\Hp}} \leq \left( \sup_{\theta \in S^{n-1}} |V(t,\theta)| \right) \|f\|_{\Hm}
\end{align*}
and so
\begin{align*}
	\|V_t\|_{B(H^{1/2+\beta}(S^{n-1}),\Hm)} &= \sup_{f \neq 0} \frac{\|V_t f\|_{\Hm}}{\|f\|_{H^{1/2+\beta}(S^{n-1})}} \leq C \sup_{\theta \in S^{n-1}} |V(t,\theta)|,
\end{align*}
where $C$ depends on the norm of the embedding $H^{1/2+\beta}(S^{n-1}) \hookrightarrow \Hm$. This proves the claimed decay estimate for $B(\tau)$.

By the same argument we obtain
\begin{align*}
	\|B(\tau_1) - B(\tau_2)\|_{B(\cH^\beta,\cH)} \leq C \sup_{\theta \in S^{n-1}} \big|t_1^2 V(t_1,\theta) - t_2^2 V(t_2,\theta) \big|.
\end{align*}
For any $0 < t_1, t_2 \leq 1$ and $\theta \in S^{n-1}$ we compute
\begin{align*}
	\big|t_1^2 V(t_1,\theta) - t_2^2 V(t_2,\theta) \big| &\leq \left| t_1^2 - t_2^2 \right| |V(t_1,\theta)| + t_2^2 |V(t_1,\theta) - V(t_2,\theta)| \\
	&\leq 2 \left| t_1 - t_2 \right| |V(t_1,\theta)| + |V(t_1,\theta) - V(t_2,\theta)|
\end{align*}
and so $\|B(\tau_1) - B(\tau_2)\|_{B(\cH^\beta,\cH)} \leq C' \left| t_1 - t_2 \right|^\gamma$. The required estimate now follows from the fact that $|t_1 - t_2| = | e^{\tau_1} - e^{\tau_2} | \leq |\tau_1 - \tau_2|$ for all $\tau_1, \tau_2 \leq 0$.
\end{proof}

\subsection{Unique continuation}
We next prove a unique continuation result for the rescaled system \eqref{RSES} and its adjoint. Given the equivalence established in Theorem \ref{thm:ODEPDE}, this is an easy consequence of the unique continuation principle for elliptic equations; see, for instance \cite{BR12}.

\begin{lemma}\label{lem:UCP}
Suppose $(\wtf,\wtg)$ is a solution of \eqref{RSES} on $(-\infty,0)$. If $(\wtf(0),\wtg(0)) = 0$, then $(\wtf(\tau),\wtg(\tau)) = 0$ for all $\tau \leq 0$.
\end{lemma}

\begin{proof}
Let $(f(t),g(t))$ denote the corresponding solution to \eqref{SES}, obtained by undoing the transformation \eqref{scaling}. Using the results of \cite{BCJLS2}, we can write $(f(t),g(t)) = \Tr_t u$, where $u \in H^1(B_1 \setminus\{0\})$ is a weak solution to $\Delta u = V u$. Then
\[
	\left.u\right|_{\p B_1} = f(1) = 0, \quad \left.\frac{\p u}{\p \nu}\right|_{\p B_1} = g(1) = 0
\]
and so $u$ must be identically zero. It follows that $f(t) = 0$ and $g(t) = 0$ for all $t \in (0,1]$, hence $\wtf(\tau)$ and $\wtg(\tau)$ vanish for $\tau \leq 0$.
\end{proof}

We also need a unique continuation result for the adjoint system
\begin{align}\label{adjoint}
    \frac{d}{d\tau} \begin{pmatrix} \wtf \\ \wtg \end{pmatrix} = \begin{pmatrix} -\alpha \  & \ \Delta_{S^{n-1}} - e^{2\tau} V_{e^\tau}  \\ -1 & n - 2 - \alpha \end{pmatrix} \begin{pmatrix} \wtf \\ \wtg \end{pmatrix}.
\end{align}
A direct calculation shows that $(f(t),g(t))$ satisfies \eqref{SES} if and only if the rescaled quantity
\begin{align}\label{rescale2}
     \begin{pmatrix} -e^{(n-1-\alpha)\tau} g(e^\tau) \\ e^{(n-2-\alpha)\tau} f(e^\tau) \end{pmatrix}
\end{align}
satisfies \eqref{adjoint}. Therefore, the adjoint system \eqref{adjoint} is also equivalent to the PDE \eqref{PDE}, in the sense of Theorem \ref{thm:ODEPDE}, and so the argument of Lemma \ref{lem:UCP} applies.

\subsection{Proof of Theorem \ref{thm:dichotomy} and Corollary \ref{cor:alphagap}}\label{sec:proofs}
Given Lemmas \ref{lem:compact}, \ref{lem:spectrum}, \ref{lem:resolvent}, \ref{lem:pert} and \ref{lem:UCP}, Theorem \ref{thm:dichotomy} is an immediate consequence of \cite[Theorem 1]{PSS97}. In fact, we are in the even better situation of \cite[Corollary 2]{PSS97}, which guarantees that $P^u(\tau)$ decays exponentially to the projection onto the unstable subspace for the autonomous operator $A$ as $\tau \to -\infty$.

To prove Corollary \ref{cor:alphagap}, suppose $0 < \alpha < n-2$, so the condition $-\eta^s < \alpha$ is satisfied for any $\eta^s \geq 0$. Moreover, the smallest positive eigenvalue of $A$ is $\alpha$, so we can choose any $\eta^u \in [0,\alpha)$. Therefore it suffices to choose $\eta^u \in (\alpha + 1 - n/2, \alpha)$. This interval is nonempty because $n>2$, and contains positive numbers because $\alpha >0$. This completes the proof of the corollary.

Finally, we prove the claim that no such $\alpha$ exists when $n=2$. To see this, let $-\alpha \notin \Sigma(2) = \bbZ$, so $\alpha + k \in (0,1)$ for some $k \in \bbZ$. The growth and decay rates must satisfy
\[
	0 \leq \eta^u < \alpha + k, \quad 0 \leq \eta^s < 1 - \alpha - k.
\]
Assuming \eqref{alpharange} holds with $n=2$, the condition $\alpha < \eta^u$ implies $k \geq 1$, hence $\eta^s < 1 - \alpha - k \leq -\alpha$, which contradicts the other inequality in \eqref{alpharange}.

As mentioned in the introduction, the non-existence of suitable $\alpha$ for $n=2$ is related to the spectrum of the asymptotic operator $A$. When $\alpha=0$ the spectrum is given by the set $\Sigma(n)$ defined in \eqref{Aspec}. Note that $0$ is always an eigenvalue of $A$, corresponding to the space of constant functions. When $n>2$ the eigenvalue $2-n$ corresponds to the fundamental solution $r^{2-n}$, which is singular at the origin. The exponential dichotomy distinguishes between these solutions provided $\alpha \in (0,n-2)$; this is precisely the content of Corollary \ref{cor:alphagap}. On the other hand, when $n=2$ the eigenvalue $0$ is repeated, on account of the harmonic function $\log r$, which blows up at the origin at a slower rate than any polynomial, in the sense that $r^\alpha \log r \to 0$ as $r \to 0$ for any $\alpha>0$.

\section{Dichotomy subspaces and spherical harmonics}\label{sec:harmonic}
We illustrate the results of the previous section for harmonic functions on $\bbR^3$. In this case $V=0$, so \eqref{RSES} becomes
\begin{align}\label{hSES}
	\frac{d}{d\tau} \begin{pmatrix} \wtf \\ \wtg \end{pmatrix} = 
	\begin{pmatrix} \alpha & 1\\ - \Delta_{S^2}  \ \ & \ \ \alpha -1 \end{pmatrix} \begin{pmatrix} \wtf \\ \wtg \end{pmatrix}.
\end{align}
In particular, $B(\tau) = 0$, so we are in the simpler case of \cite[Lemma 2.1]{PSS97}, which guarantees the existence of a dichotomy for \eqref{RSES} on the entire real line, with $\tau$-independent projections $P^s$ and $P^u$.

\subsection{The dichotomy subspaces}
From Lemma \ref{lem:spectrum} the eigenvalues of $A$ are
\[
	\nu^+_l = \alpha + l, \quad \nu^-_l = \alpha - l - 1,
\]
for $l=0,1,2,\ldots$. Each $\nu_l^\pm$ has multiplicity $2l+1$. The eigenfunctions can be expressed in terms of spherical harmonics $Y^m_l$ as
\begin{align*}
	\begin{pmatrix} \wtf^+_{lm}(\tau) \\ \wtg^+_{lm}(\tau)\big) \end{pmatrix} = e^{(\alpha + l) \tau } Y^m_l \begin{pmatrix} 1 \\ l \end{pmatrix}, \quad
	\begin{pmatrix} \wtf^-_{lm}(\tau) \\ \wtg^-_{lm}(\tau)\big) \end{pmatrix} = e^{(\alpha - l - 1) \tau} Y^m_l \begin{pmatrix} 1 \\  -(l+1) \end{pmatrix}
\end{align*}
for $-l \leq m \leq l$, and so
\begin{align*}
	\begin{pmatrix} f^+_{lm}(t) \\ g^+_{lm}(t) \end{pmatrix} = t^l Y^m_l \begin{pmatrix} 1 \\ l/t \end{pmatrix}, \quad
	\begin{pmatrix} f^-_{lm}(t) \\ g^-_{lm}(t) \end{pmatrix} = t^{-l-1} Y^m_l \begin{pmatrix} 1\\ -(l+1)/t \end{pmatrix}.
\end{align*}
Note that $\big(f^+_{lm}(t), g^+_{lm}(t) \big)$ is the boundary data of the harmonic function $u(r,\theta,\phi) = r^l Y^m_l(\theta,\phi)$ on the surface $\{|x|=t\}$, and $\big(f^-_{lm}(t), g^-_{lm}(t) \big)$ is the boundary data of $u(r,\theta,\phi) = r^{-l-1} Y^m_l(\theta,\phi)$. Solutions corresponding to $\nu^+_l$ are bounded at the origin and blow up at infinity, whereas solutions corresponding to $\nu^-_l$ blow up at the origin and decay to zero at infinity.

The unstable subspace $\widetilde E^u(\tau)$ is spanned by the eigenfunctions for which the corresponding eigenvalue $\nu_l^\pm$ is positive, and similarly for the stable subspace $\widetilde E^s(\tau)$. For any $\alpha \in (0,1)$ we have $\nu^-_0 < 0 < \nu^+_0$, and hence $\nu^-_l < 0 < \nu^+_l$ for all $l$. Therefore, for any such $\alpha$, $E^u(t)$ is precisely the set of boundary data of harmonic functions that are bounded at the origin, as was shown more generally in Corollary~\ref{cor:alphagap}.

\subsection{The dichotomy projections}
We assume the spherical harmonics $Y_l^m$ are normalized so that $\left<Y_l^m, Y_k^n\right> = \delta_{mn} \delta_{lk}$, 
where $\left<\cdot,\cdot\right>$ denotes the $L^2(S^2)$ inner product:
\[
 \langle f,g \rangle = \int_0^{2\pi} \int_0^\pi f(\theta, \phi) g(\theta, \phi) \sin\theta \,d\theta d\phi.
\]
Expanding $z = (z_1,z_2) \in \cH$ as
\[
    z = \sum_{l=0}^\infty \sum_{m=-l}^l \left( A_{lm} \begin{pmatrix} 1 \\ l \end{pmatrix} + B_{lm} \begin{pmatrix} 1 \\ -(l+1) \end{pmatrix} \right) Y_l^m,
\]
we find that
\[
    A_{lm} = \frac{1}{2l+1} \left<(l+1)z_1 + z_2, Y^m_l \right>, \quad
    B_{lm} = \frac{1}{2l+1} \left<lz_1 - z_2, Y^m_l \right>,
\]
and so the dichotomy projections are given by
\begin{align}
    P^u z &= \sum_{l=0}^\infty \sum_{m=-l}^l Y_l^m \frac{1}{2l+1} \left<(l+1)z_1 + z_2, Y^m_l \right> \begin{pmatrix} 1 \\ l \end{pmatrix}, \\
    P^s z &= \sum_{l=0}^\infty \sum_{m=-l}^l Y_l^m \frac{1}{2l+1} \left<l z_1 - z_2, Y^m_l \right> \begin{pmatrix} 1 \\ -(l+1) \end{pmatrix}.
\end{align}

\subsection{The evolution operators}
We next give explicit formulas for the operators $\Phi^{s,u}(\tau,\tau_0)$ defined in \eqref{Phidef}.

For arbitrary $z \in \cH$, $\Phi^{u}(\tau,\tau_0)z$ must be of the form $\sum\sum C_{lm}  e^{(\alpha+l)\tau} \big(Y_l^m, l Y_l^m\big)$. Using the formula for $P^u$ obtained above, and the fact that $\Phi^u(\tau_0,\tau_0)z = P^u z$, we find that $C_{lm} = e^{-(\alpha+l)\tau_0} \left<(l+1)z_1 + z_2, Y_l^m\right> / (2l+1)$, and hence
\begin{align}\label{Phiu}
    \Phi^u(\tau,\tau_0)z = \sum_{l = 0}^\infty \sum_{m = -l}^l e^{(\alpha + l)(\tau-\tau_0)} Y^m_l \frac{1}{(2l+1)}\langle (1+l)z_1 + z_2, Y^m_l \rangle \begin{pmatrix} 1 \\ l \end{pmatrix}
\end{align}
for $\tau \leq \tau_0$. Similarly, we obtain
\begin{align}\label{Phis}
    \Phi^s(\tau,\tau_0)z = \sum_{l = 0}^\infty \sum_{m = -l}^l e^{(\alpha - l - 1)(\tau-\tau_0)} Y^m_l \frac{1}{(2l+1)}\langle l z_1 - z_2, Y^m_l \rangle \begin{pmatrix} 1 \\ -(l+1) \end{pmatrix}
\end{align}
for $\tau \geq \tau_0$.

\subsection{Liouville-type theorems}
Since \eqref{hSES} is autonomous, the exponential dichotomy exists on the entire real line; cf. Theorem \ref{thm:dichotomy} which only guarantees the existence of a half-line dichotomy. Therefore, \cite[Theorem 2]{PSS97} says that the only bounded solution to \eqref{hSES} is $(\wtf(\cdot), \wtg(\cdot)) = (0,0)$. Using this, we obtain the following Liouville-type result, which rules out the existence of slowly-growing harmonic functions.

\begin{corollary}
Suppose $u$ is an entire harmonic function on $\bbR^n$. If $\|u\|_{H^1(\Omega_t)} \leq C t^r$ for some $r < n/2 - 1$, then $u$ is identically zero.
\end{corollary}

\begin{proof}
From \cite{BCJLS2} we have the estimates
\begin{align*}
    \|f(t)\|_{\Hp} &\leq C t^{-n/2} \|u\|_{H^1(\Omega_t)}, \\
    \|g(t)\|_{\Hm} &\leq C t^{-n/2} \|u\|_{H^1(\Omega_t)},
\end{align*}
and hence
\begin{align*}
    \|\wtf(\tau)\|_{\Hp} &\leq C t^{\alpha-n/2} \|u\|_{H^1(\Omega_t)}, \\
    \|\wtg(\tau)\|_{\Hm} &\leq C t^{1 + \alpha-n/2} \|u\|_{H^1(\Omega_t)}.
\end{align*}
Choose a number $0 < \alpha < (n/2-1) - r$ with $-\alpha \notin \Sigma(n)$, so that Theorem \ref{thm:dichotomy} applies. It follows from elliptic regularity that $u$ and $\nabla u$ are uniformly bounded in a neighborhood of the origin, say $|u(x)|,|\nabla u(x)| \leq c$ for all $x \in \Omega_t$, with $t$ sufficiently small. Then
\[
    \|u\|_{H^1(\Omega_t)}^2 = \int_{\Omega_t} \big(|u|^2 + |\nabla u|^2 \big) \leq 2 c^2 |\Omega_t| = 2 c^2 \omega_n t^n,
\]
and so $\|u\|_{H^1(\Omega_t)} \leq C t^{n/2}$ for small $t$. Since $\alpha>0$, both $\|\wtf(\tau)\|_{\Hp}$ and $\|\wtg(\tau)\|_{\Hm}$ are thus bounded as $\tau \to -\infty$. On the other hand, the hypothesis $\|u\|_{H^1(\Omega_t)} \leq C t^r$ implies $$\|\wtg(\tau)\|_{\Hm} \leq C t^{1 + \alpha-n/2 + r}$$ is bounded as $\tau \to \infty$, since $1 + \alpha-n/2 + r<0$, and similarly for $\|\wtf(\tau)\|_{\Hp}$.
\end{proof}

\section{Applications}\label{sec:apply}
The previous sections gave a dynamical interpretation of the linear elliptic equation \eqref{PDE}, expanding on the results in \cite{BCJLS2} in the radial case. We conclude by presenting some applications of these ideas to linear and nonlinear PDE. In particular, we show that the presence (or absence) of unstable eigenvalues is encoded in the dichotomy subspaces, and demonstrate how the exponential dichotomy can be used to construct solutions to nonlinear equations on bounded and unbounded domains.

\subsection{Eigenvalue problems}

Here we use Corollary \ref{cor:alphagap} to give a dynamical interpretation of the eigenvalue problem
\begin{align}\label{eigenvalue}
	-\Delta u + V u = \lambda u
\end{align}
with Dirichlet boundary conditions. To do so we let $E^u(t)$ denote the unstable subspace corresponding to \eqref{eigenvalue}, with $\alpha$ chosen to satisfy the hypotheses of Theorem~\ref{thm:unstable}, and define the \emph{Dirichlet subspace}
\begin{align}\label{Dirsubspace}
	\cD = \{( 0,g) : g \in \Hm\} \subset \cH.
\end{align}

\begin{theorem}
$\lambda$ is an eigenvalue of the Dirichlet problem \eqref{eigenvalue} on $B_t$ if and only if the unstable subspace $E^u(t)$ intersects the Dirichlet subspace $\cD$ nontrivially. Moreover, the multiplicity of $\lambda$ equals $\dim \big(E^u(t) \cap \cD\big)$.
\end{theorem}

Other boundary conditions (Neumann, Robin, etc.) can be characterized in a similar way by replacing $\cD$ accordingly; see \cite{CJLS16,CJM15} for details.

Therefore we have given a dynamical perspective on elliptic eigenvalue problems, similar to the Evans function \cite{S02}, which counts intersections between stable and unstable subspaces. This is also closely related to the Maslov index, a symplectic winding number that counts intersections of Lagrangian subspaces in a symplectic Hilbert space; see \cite{CJLS16,CJM15,DJ11,LS18}.

\subsection{Reformulation of two nonlinear problems}

In this section we illustrate how to reformulate equations of the form
\[
\Delta u - V(x) u = F(x, u),
\]
where $F$ is smooth with $F(x, 0) = D_u F(x, 0) = 0$, using the dichotomy constructed above. 

We emphasize that this approach allows for the construction of solutions that are not radially symmetric, even though spherical subdomains $\Omega_t = \{ |x| < t\}$ are used in constructing the dichotomy.

\subsubsection{A nonlinear boundary value problem}

First, we consider the case where $x \in B_T = \{x \in \bbR^n : |x| < T\}$, with some appropriate boundary condition:
\begin{align}\label{E:bounded-domain}
    \begin{split}
   \Delta u - V(x) u = F(x, u), \quad x \in B_T  \\
\big(u|_{\partial B_T},\p_\nu u|_{\partial B_T}\big) \in \mathcal{B}, 
    \end{split}
\end{align}
for some subspace $\mathcal{B} \subset \cH$. Using the framework introduced above, we write this as the equivalent spatial evolutionary system
\begin{equation}\label{E:SES-nonlinear}
\frac{d}{dt} \begin{pmatrix} f \\ g \end{pmatrix} = \begin{pmatrix} 0 & 1 \\ V(t, \theta) - t^{-2}\Delta_{S^{n-1}} \ \ & \ \ -t^{-1}(n-1)\end{pmatrix} \begin{pmatrix} f \\ g \end{pmatrix}  +  \begin{pmatrix} 0 \\ F(t, \theta, f) \end{pmatrix}. 
\end{equation}
Applying the change of variables used above, $t = e^\tau$, $\tilde f(\tau) = e^{\alpha \tau}f(e^\tau)$, $\tilde g(\tau) = e^{(\alpha+1) \tau}g(e^\tau)$, we find
\begin{equation}\label{E:SES-nonlinear-rescaled}
\frac{d}{d\tau} \begin{pmatrix} \tilde f \\ \tilde g \end{pmatrix} = \begin{pmatrix} \alpha & 1 \\ e^{2\tau} V(e^\tau, \theta) - \Delta_{S^{n-1}} \ \ & \ \ \alpha + 2 -n \end{pmatrix} \begin{pmatrix} \tilde f \\ \tilde g \end{pmatrix}  +  \begin{pmatrix} 0 \\ e^{(\alpha+2)\tau} F(e^\tau, \theta, e^{-\alpha \tau} \tilde f) \end{pmatrix}. 
\end{equation}

It was shown above that for any $\beta \in [0,1)$ an exponential dichotomy exists in $\cH^\beta$ on the interval $(-\infty, \log T]$, for the linear evolution associated with the above system, as long as 
$-\alpha \notin \Sigma(n)$ and $V \in C^{0, \gamma}(\Omega)$, which we assume in this section. For notational convenience, write the above system as
\begin{equation} \label{E:compact-notation}
    \frac{d}{d\tau} \wth = \mathcal{A}(\tau)\wth + \mathcal{F}(\tau, \wth), \qquad \wth = \begin{pmatrix} \wtf \\ \wtg \end{pmatrix}, 
\end{equation} 
where
\[
\mathcal{A}(\tau) = \begin{pmatrix} \alpha & 1 \\ e^{2\tau} V(e^\tau, \theta) - \Delta_{S^{n-1}} \ \ & \ \ \alpha + 2 -n \end{pmatrix}, \quad \mathcal{F}(\tau, \wth) = \begin{pmatrix} 0 \\ e^{(\alpha+2)\tau} F(e^\tau, \theta, e^{-\alpha \tau} \tilde f) \end{pmatrix},
\]
and we have notationally suppressed any $\theta$-dependence. With a suitable assumption on the nonlinearity $F$, any solution to \eqref{E:compact-notation} that is bounded as $\tau \to -\infty$ can be written in terms of the operators $\Phi^{s,u}$ defined in \eqref{Phidef} as
\begin{equation}\label{E:soln-minus-infty}
    \wth(\tau) = \Phi^u(\tau, \log T) \wth_* + \int_{-\infty}^\tau \Phi^s(\tau, \rho)\mathcal{F}(\rho, \wth(\rho)) d \rho + \int_{\log T}^\tau \Phi^u(\tau, \rho) \mathcal{F}(\rho, \wth(\rho)) d \rho
\end{equation}
for some $\wth_* \in \cH^\beta$. For instance, it is sufficient to have $\mathcal F \in C^{1,1} \big((-\infty,\log T] \times \cH^\beta,\cH \big)$, which is equivalent to requiring that the map $(\tau, \tilde f) \mapsto e^{(\alpha+2)\tau} F(e^\tau, \theta, e^{-\alpha \tau} \tilde f)$ is in $C^{1,1} \big((-\infty,\log T] \times H^{1/2 + \beta}(S^{n-1}),\Hm \big)$; see \cite[p. 294]{PSS97}.

Using the fact that 
\[
\frac{d}{d\tau} \Phi^{s,u}(\tau, \rho) = \mathcal{A}(\tau)\Phi^{s,u}(\tau, \rho), \qquad \Phi^s(\tau, \tau) + \Phi^u(\tau, \tau) = \mathrm{Id}, 
\]
once can directly check that $\wth(\tau)$ given in \eqref{E:soln-minus-infty} is indeed a solution of \eqref{E:compact-notation}. The exponential bounds for $\Phi^{s,u}(\tau, \rho)$ ensure that it is well-behaved as $\tau \to -\infty$. At the moment, $\wth_* \in \cH^\beta$ is arbitrary. However, we have not yet made reference to the boundary condition. We need
\begin{equation}\label{E:BC}
    \wth(\log T) = P^u(\log T) \wth_* + \int_{-\infty}^{\log T} \Phi^s(\log T, \rho)\mathcal{F}(\rho, \wth(\rho)) d \rho  \in \mathcal{B}.
\end{equation}
The idea is thus to choose $\wth_* \in \cH^\beta$ so that \eqref{E:BC} holds. Note that, since $\wth$ is defined implicitly via \eqref{E:soln-minus-infty}, the integral term in \eqref{E:BC} depends on the choice of $\wth_*$ through $\wth$. The best way to understand \eqref{E:BC} would depend on the details of the dichotomy and the boundary conditions. 

\subsubsection{A nonlinear problem on $\mathbb{R}^n$}

Next consider
\begin{equation}\label{E:unbounded-domain}
\Delta u - V(x) u = F(x, u), \qquad x \in \mathbb{R}^n.
\end{equation}
If we reformulate this as the evolutionary system \eqref{E:SES-nonlinear-rescaled}, then the linear part admits an exponential dichotomy on the negative half line $(-\infty,0]$, by Theorem \ref{thm:dichotomy}. We denote this by $\Phi^{s,u}_-$. Moreover, if $|x|^2 V(x) \to 0$ as $|x| \to \infty$, the proof of Theorem \ref{thm:dichotomy} also yields a dichotomy on the positive half line $[0,\infty)$, which we denote by $\Phi^{s,u}_+$. (When $V=0$ we have a dichotomy on the whole line, so $\Phi^u_\pm$ and $\Phi^s_\pm$ are given explicitly by \eqref{Phiu} and \eqref{Phis} for $n=3$, and can be expressed similarly for $n>3$.)

As in the previous section, with a suitable assumption on the nonlinearity $F$, bounded solutions on $(-\infty, 0]$ are given by
\begin{equation}\label{E:soln-minus-infty-2}
    \wth^-(\tau) = \Phi_-^u(\tau, 0) \wth_1 + \int_{-\infty}^\tau \Phi_-^s(\tau, \rho)\mathcal{F}(\rho, \wth^-(\rho)) d \rho + \int_{0}^\tau \Phi_-^u(\tau, \rho) \mathcal{F}(\rho, \wth^-(\rho)) d \rho
\end{equation}
and bounded solutions on $[0, \infty)$ are given by
\begin{equation}\label{E:soln-plus-infty-2}
    \wth^+(\tau) = \Phi_+^s(\tau, 0) \wth_2 + \int_{+\infty}^\tau \Phi_+^u(\tau, \rho)\mathcal{F}(\rho, \wth^+(\rho)) d \rho + \int_{0}^\tau \Phi_+^s(\tau, \rho) \mathcal{F}(\rho, \wth^+(\rho)) d \rho,
\end{equation}
where $\wth_{1,2} \in \cH^\beta$ are, for the moment, arbitrary.

To find a solution to \eqref{E:SES-nonlinear-rescaled} that is bounded for all $\tau \in \mathbb{R}$, we must match \eqref{E:soln-minus-infty-2} and \eqref{E:soln-plus-infty-2} at $\tau = 0$. This leads to the matching condition
\begin{eqnarray*}
    0 &=& \wth^+(0) - \wth^-(0) \\ 
    &=& P_-^u(0)\wth_1+ \int_{-\infty}^{0} \Phi_-^s(0, \rho)\mathcal{F}(\rho, \wth^-(\rho)) d \rho \\
    &&\qquad - P_+^s(0) \wth_2 -  \int_{+\infty}^{0} \Phi_+^u(0, \rho)\mathcal{F}(\rho, \wth^+(\rho)) d \rho.
\end{eqnarray*}
Similar to the previous example, the best way to understand this matching condition depends on the details of the nonlinearity. In the $V=0$ case one has the advantage of having an explicit formula for the dichotomy and the projection operators.

\begin{acknowledgement}
The authors would like to acknowledge the support of the American Institute of Mathematics and the Banff International Research Station, where much of this work was carried out. M.B. acknowledges the support of NSF grant DMS-1411460 and of an AMS Birman Fellowship. G.C. acknowledges the support of NSERC grant RGPIN-2017-04259. C.J. was supported by ONR grant N00014-18-1-2204. Y.L. was supported by the NSF grant DMS-1710989, by the Research Board and Research Council of the University of Missouri, and by the Simons Foundation. A.S. was supported by NSF grant DMS-1910820.
\end{acknowledgement}

\bibliographystyle{spmpsci}
\bibliography{proceedings.bib}

\end{document}